\documentclass[a4paper]{article}
\usepackage[all]{xy}\usepackage[latin1]{inputenc}        %accents
\usepackage[dvips]{graphics,graphicx}
\usepackage{amsfonts,amssymb,amsmath,color,mathrsfs, amstext}
\usepackage{amsbsy, amsopn, amscd, amsxtra, amsthm,authblk}
\usepackage{enumerate,algorithmicx,algorithm}
\usepackage{algpseudocode}
\usepackage{upref}
\usepackage{geometry}
\usepackage[displaymath]{lineno}
\usepackage{float}
\usepackage{yhmath}
\usepackage{booktabs}
\usepackage{subcaption}
\usepackage{multirow}
\usepackage{makecell}
\usepackage[colorlinks,
            linkcolor=red,
            anchorcolor=red,
            citecolor=red
            ]{hyperref}

\numberwithin{equation}{section}

\newcommand{\E}{\mathbb{E}}
\newcommand{\R}{\mathbb{R}}
\newcommand{\bbP}{\mathbb{P}}

\newcommand{\cL}{\mathcal{L}}

\newtheorem{theorem}{Theorem}[section]
\newtheorem{lemma}{Lemma}[section]

\newtheorem{remark}{Remark}[section]
\newtheorem{proposition}{Proposition}[section]
\newtheorem{assumption}{Assumption}[section]

\numberwithin{equation}{section}

\begin{document}

\title{Convergence of Random Batch Method for interacting particles with disparate species and weights}

\author[1]{Shi Jin \thanks{shijin-m@sjtu.edu.cn}}
\author[2]{Lei Li\thanks{leili2010@sjtu.edu.cn}}
\author[3]{Jian-Guo Liu\thanks{jliu@phy.duke.edu}}
\affil[1,2]{School of Mathematical Sciences, Institute of Natural Sciences, MOE-LSC, Shanghai Jiao Tong University, Shanghai, 200240, P. R. China.}
\affil[3]{Department of Mathematics and Department of Physics, Duke University, Durham, NC 27708, USA.}

\date{}
\maketitle

\begin{abstract}
We consider in this work the convergence of Random Batch Method proposed in our previous work [Jin et al., J. Comput. Phys., 400(1), 2020]  for interacting particles to the case of disparate species and weights. We show that the strong error is of $O(\sqrt{\tau})$ while the weak error is of $O(\tau)$ where $\tau$ is the time step between two random divisions of batches. Both types of convergence are uniform in $N$, the number of particles. The proof of strong convergence follows closely the proof in  [Jin et al., J. Comput. Phys., 400(1), 2020]  for indistinguishable particles, but there are still some differences: since there is no exchangeability now, we have to use a certain weighted average of the errors; some refined auxiliary lemmas have to be proved compared with our previous work. To show that the weak convergence of empirical measure is uniform in $N$, certain sharp estimates for the derivatives of the backward equations have been used. The weak convergence analysis is also illustrating for the convergence of Random Batch Method for $N$-body Liouville equations. 
\end{abstract}

\section{Introduction}

Interacting particle systems are ubiquitous in natural, for example, molecules in fluids \cite{frenkel2001understanding}, plasma \cite{birdsall2004}, galaxy in universe. In addition, many collective behaviors in natural and social sciences are due to interacting individuals, and examples include swarming \cite{vicsek1995novel,carrillo2017particle,carlen2013kinetic,degond2017coagulation}, flocking  \cite{cucker2007emergent,hasimple2009,albi2013}, and 
chemotaxis \cite{horstmann03,bertozzi12} and consensus clusters in opinion dynamics \cite{motsch2014}. In many models for these phenomenon, individual particles can have weights. For example, in the point vortex model \cite{chorin1973,gustafson1991,liu2001convergence}, the ``particles" correspond to different point vortices with different strength and they interact each other through a Hamiltonian system. To approximate the nonlinear Fokker-Planck equation for interacting particles, like the Keller-Segel equation, one can use interacting particles with different masses to approximate the dynamics and then compute the empirical density more efficiently  \cite{craig16,liu2017random}. There may also be several species, where the particles may have different features; for example,  in the microscopic description of the Poisson-Boltzmann equation, particles with different charges interact with each other through Coulomb forces \cite{allaire2013asymptotic,jin2020direct}.

The systems mentioned above can be written as the second order ODE/SDE system for $1\le i\le N$
\begin{gather}\label{eq:sdeeq1}
\begin{split}
&dX^i= V^i\,dt\\
& dV^i=\frac{1}{N-1}\sum_{j=1, j\neq i}^N q_j F(X^i-X^j)\,dt-\gamma V^i\,dt+\sigma\, dW^i
\end{split}
\end{gather}
or the first order ODE/SDE system
\begin{gather}\label{eq:sdeeq}
dX^i=\frac{1}{N-1}\sum_{j=1: j\neq i}^N q_j F(X^i-X^j)\,dt+\sigma\, dW^i,~~i=1,\cdots, N,
\end{gather}
where $q_j$'s are the weights (they can be mass, charges etc). Note that \eqref{eq:sdeeq1} contains the Hamiltonian case when the force $F$ is conservative and $\gamma=\sigma=0$.
If one discretizes \eqref{eq:sdeeq1} or \eqref{eq:sdeeq} directly, the computational cost per time step is $O(N^2)$, which is undesired for large $N$. In the case of fast enough decaying interactions, the Fast Multipole Method (FMM) \cite{rokhlin1985rapid} can reduce the complexity per iteration to $O(N)$. However, the implementation is not easy, as it needs some advanced data structures.  

The authors proposed in \cite{jin2020random}  a simple random algorithm, called the Random Batch Method (RBM), to reduce the computation cost per time step from $O(N^2)$ to $O(N)$ for indistinguishable particles (thus with the same weight). The idea was to apply the ``mini-batch" idea, famous for its application in the so-called stochastic gradient descent (SGD) \cite{robbins1951,bottou1998online,bubeck2015convex} in machine learning. Later on, the ``mini-batch" was used to develop a Markov Chain Monte Carlo method,  called the stochastic gradient Langevin dynamics (SGLD), by Welling and Teh for Bayesian inference  \cite{welling2011bayesian}. The random binary collisions were also used to simulate Boltzmann equation \cite{bird1963approach,nanbu1980direct,babovsky1989convergence} or the mean-field equation for flocking \cite{albi2013}.  How to design the mini-batch strategy depends on the specific applications. For interacting particle systems, the strategy in \cite{jin2020random} is to do random grouping at each time interval and then let the particles interact within the groups on that small time interval. Compared with FMM, the accuracy is lower (halfth order in time step), but RBM is simpler to implement and is valid for more general potentials (e.g. the SVGD ODE \cite{liu2016stein,li2019stochastic}). 
Intuitively, the method converges because there is time average in time, and thus the convergence is like that in the Law of Large Number, but in time (see \cite{jin2020random}  for more detailed explanation). Moreover, the RBM converges in the regime of mean field limit (\cite{stanley1971, georges1996, lasry2007}). In fact, the method is asymptotic-preserving regarding the mean-field limit, which means the algorithm can approximate the one-marginal distribution with error bound independent of $N$.

The proof in \cite{jin2020random} depends on the exchangeability of the particles. If the particles carry weights, they are not exchangeable. 
Even so, it is clear that RBM can be equally well applied for interacting particles in multispecies and with weights  (see Algorithm \ref{rbm} below). The goal of this paper is to analyze the convergence of RBM for interacting particles with weights or multispecies. We will discuss both strong and weak convergences, which rely on different techniques so that the error estimates are independent of $N$.  The proof of strong convergence follows closely the proof in \cite{jin2020random} for indistinguishable particles, but there are still some differences: since there is no exchangeability, we have to use a kind of weighted average of the errors; some refined auxiliary lemmas have to be used here, compared with those in \cite{jin2020random}. The weak convergence is for the empirical measures instead of the one marginal distribution as in \cite{jin2020random}. Some sharp estimates for the derivatives of the backward equations have to be done so that the weak convergence is  uniform in $N$. The weak convergence analysis is also illustrating for the convergence of Random Batch Method for $N$-body Liouville equations. 

The rest of the paper is organized as follows. In section \ref{sec:setup},
we introduce some basic notations and give a detailed description of the Random Batch Method. Section \ref{sec:strong} is devoted to the proof of strong convergence. In section \ref{sec:weak}, we show the weak first order convergence, where a key proposition for the estimates of derivatives of the backward equation  is needed.

\section{Setup and notations}\label{sec:setup}

We consider interacting particles with weights. Since in practice, there
can be external field, and also there are some applications where the interaction kernels depend on the two specific particles (for example, the SVGD ODE \cite{li2019stochastic}), we consider in general the following first order system 
\begin{gather}\label{eq:Nparticlesys}
dX^i=b(X^i)\,dt+\frac{1}{N-1}\sum_{j: j\neq i} m_jK_{ij}(X^i, X^j)\,dt+\sigma dW^i,~~i=1,2,\cdots, N.
\end{gather}
The argument in this paper can be generalized to second order systems without difficulty, which we omit.
In \eqref{eq:Nparticlesys}, $b(\cdot)$ is the external force field, $\{W^i\}'s$ are some given independent $d$ dimensional Wiener processes (the standard Brownian motions) and we impose
\begin{gather}
m_j\ge 0.
\end{gather}
Note that this model includes the cases for particles with multispecies since the signs of the interaction can be included into $K_{ij}$ and for \eqref{eq:sdeeq} $m_j=|q_j|$. 

For notational convenience, we define
\begin{gather}
F_i(x):=b(x_i)+\frac{1}{N-1}\sum_{j: j\neq i}m_j K_{ij}(x_i, x_j),
\end{gather}
where 
\[
x:=(x_1, \cdots, x_N)\in \R^{Nd}.
\]

The (random) empirical probability measure corresponding  to \eqref{eq:Nparticlesys} is given by
\begin{gather}\label{eq:empiricalmeasure}
\mu^N:=\frac{1}{N}\sum_{j=1}^N \omega_j \delta(x-X^j),
\end{gather}
where
\[
\omega_j=\frac{Nm_j}{\sum_j m_j}.
\]

\begin{assumption}\label{ass:mass}
We assume there are positive constants $A,M$ independent of $N$ such that
\begin{gather}
\frac{1}{N}\sum_j m_j=M,~~\max_j|m_j|\le A.
\end{gather}
\end{assumption}
With the assumption above, we find
\begin{gather}
\omega_j=O(1).
\end{gather}

Below are some assumptions on the external field and interaction kernels.
\begin{assumption}\label{ass:kernelfunctions}
Moreover, we assume $b(\cdot)$ is one-sided Lipschitz:
\begin{gather}
(z_1-z_2)\cdot (b(z_1)-b(z_2))\le \beta |z_1-z_2|^2
\end{gather}
for some constant $\beta$ and that $b, \nabla b$ have polynomial growth
\begin{gather}
|b(z)|+|\nabla b|\le C(1+|z|)^q. 
\end{gather}
The functions $K_{ij}(\cdot, \cdot)$ and their derivatives up to second order are uniformly bounded in $1\le i,j\le N$.
\end{assumption}

\subsection{The Random Batch Method and mathematical setup}

We now give some detailed explanation of the random batch method proposed in \cite{jin2020random}, when applied on \eqref{eq:Nparticlesys}. Suppose the computational interval is $[0, T]$. We pick a small time step $\tau$, and define the discrete time grids
\begin{gather}
t_k:=k\tau,~~k\in \mathbb{N}.
\end{gather}
The number of iteration for the algorithm is
\begin{gather}\label{eq:NT}
N_T:=\left\lceil \frac{T}{\tau} \right\rceil.
\end{gather}
At each time grid $t_k$, we divide the 
\[
N=np
\]
 particles into $n$ small batches with equal size $p$ ($p\ll N$, often $p=2$) randomly. We have assumed $p$ divides $N$ for convenience. Denote the $n$ batches by $\mathcal{C}_q, q=1,\cdots, n$, and then each particle only interacts particles within its own batch. The detail is shown in Algorithm \ref{rbm}). Clearly, each iteration contains two main steps: (1) Randomly dividing the particles into $n$ batches (implemented by random permuation, costing $O(N)$ \cite{durstenfeld1964}); (2) particles interact inside batches only.

\begin{algorithm}[H]
\caption{(RBM)}\label{rbm}
\begin{algorithmic}[1]
\For {$k \text{ in } 1: [T/\tau]$}   
\State Divide $\{1, 2, \ldots, pn\}$ into $n$ batches randomly.
     \For {each batch  $\mathcal{C}_q$} 
     \State Update $\tilde{X}^i$'s ($i\in \mathcal{C}_q$) by solving the following SDE with $t\in [t_{k-1}, t_k)$.
     \begin{gather}\label{eq:firstalgorithm}
            d\tilde{X}^i=b(\tilde{X}^i) dt+\frac{1}{p-1}\sum_{j\in \mathcal{C}_q,j\neq i}m_j K_{ij}(\tilde{X}^i,\tilde{X}^j)dt+\sigma dW^i.
      \end{gather}
      \EndFor
\EndFor
\end{algorithmic}
\end{algorithm}
Above, the Wiener process $W^i$ (Brownian motion) used is the same as in \eqref{eq:Nparticlesys}.

\iffalse
\begin{remark}
One alternative strategy is that for each particle $i$, one associates it with a random set $B_i$ and then run the particles with force
\[
\frac{1}{p-1}\sum_{j: j\in B_i,j\neq i}m_j K_{ij}(\tilde{X}^i, \tilde{X}^j).
\]
The difference with Algorithm \ref{rbm} is that for some $j\in B_i$, it may happen that $B_i\neq B_j$. For example,
\[
\dot{\tilde{X}}^1=K(\tilde{X}^1-\tilde{X}^2),~~\dot{\tilde{X}}^2=K(\tilde{X}^2-\tilde{X}^3),~~\dot{\tilde{X}}^3=K(\tilde{X}^3-\tilde{X}^2).
\]
This strategy, we believe, is still convergent, but it can bring some troubles due to the symmetry breaking. This is evident for the random batch version of $N$-body Schr\"odinger equation considered in \cite{golse2019random}.
\end{remark}
\fi

\begin{remark}
For particles with weights, it is desirable to get some random batches by importance sampling. This is left for future study.
\end{remark}

We denote $\mathcal{C}_q^{(k)}$, $1\le q\le n$ the batches at $t_k$, and
\begin{gather}
\mathcal{C}^{(k)}:=\{\mathcal{C}_1^{(k)}, \cdots, \mathcal{C}_n^{(k)}\},
\end{gather}
will denote the random division of batches at $t_{k}$.
It is standard by the Kolmogorov extension theorem \cite{durrett2010} that there exists a probability space $(\Omega, \mathcal{F}, \bbP)$ so that the random variables $\{X_0^{i}, W^i, \mathcal{C}^{(k)}: 1\le i\le N, k\ge 0\}$ are on this probability space and they are all independent.
As usual, we use $\E$ to denote the integration on $\Omega$ with respect the probability measure $\bbP$.
For the convenience of the analysis, we introduce the $L^2(\bbP)$
norm as:
\begin{gather}
\|v\|=\sqrt{\mathbb{E}|v|^2}.
\end{gather}
Define the filtration $\{\mathcal{F}_{k}\}_{k\ge 0}$ by
\begin{gather}
\mathcal{F}_{k-1}=\sigma(X_0^i, W^i(t), \mathcal{C}^{(j)}; t\le t_{k-1}, j\le k-1).
\end{gather}
In other words, $\mathcal{F}_{k-1}$ is the $\sigma$-algebra generated by the initial values $X_0^i$ ($i=1,\ldots, N$), $B^i(t)$, $t\le t_{k-1}$, and $\mathcal{C}^{(j)}$, $j\le k-1$. Hence, $\mathcal{F}_{k-1}$ contains the information of how batches are constructed for $t\in [t_{k-1}, t_k)$. We also introduce the filtration $\{\mathcal{G}_k\}_{k\ge 0}$ by
\begin{gather}
\mathcal{G}_{k-1}=\sigma(X_0^i, W^i(t), \mathcal{C}^{(j)}; t\le t_{k-1}, j\le k-2).
\end{gather}
If we use $\sigma(\mathcal{C}^{(k-1)})$ to mean the $\sigma$-algebra generated by $\mathcal{C}^{(k-1)}$, the random division of batches at $t_{k-1}$, then $\mathcal{F}_{k-1}=\sigma(\mathcal{G}_{k-1}\cup \sigma(\mathcal{C}^{(k-1)}))$.

For further discussion, given some random batches $\mathcal{C}$, we define the random variables
\begin{gather}
I_{ij}=
\begin{cases}
1 & i, j\text{~are in the same batch,}\\
0  & \text{~otherwise~}
\end{cases}
, ~~1\le i, j\le N.
\end{gather}

We will focus on the approximation error of $\tilde{X}$ for $X$ for $t\in [0, T]$. In particular, we define the error process
\begin{gather}\label{eq:Zidef}
Z^i=\tilde{X}^i-X^i,~~i=1,\cdots, N.
\end{gather}

\subsection{A comment about interacting particles with multispecies}

In applications, the most important cases where particles carry weights are the multispecies cases. For example, when we simulate the microscopic particles for the Poisson-Boltzmann equations \cite{allaire2013asymptotic,jin2020direct}, we need to consider charged particles with different valences, in particular
\begin{gather}
dX^i=\frac{1}{N-1}\sum_{j=1: j\neq i}^N Q_j z_iz_j F(X^i-X^j)\,dt+\sigma\, dW^i,~~i=1,\cdots, N,
\end{gather}
where $z_i=\pm 1$ represents whether the charge is positive or negative
and $Q_j\ge 0$ is the absolute value of the charges. In this case, we can define
\begin{gather}
K_{ij}(x, y)=z_iz_j F(x-y),
\end{gather}
and this reduces to \eqref{eq:Nparticlesys}.

Also, people may care about different densities for different species. Similar to \eqref{eq:empiricalmeasure}, one can compute the empirical measures for all these species separately. The empirical measure \eqref{eq:empiricalmeasure} is then a mixture of them. Hence, the model \eqref{eq:Nparticlesys} is rich enough to include interacting particles with disparate species and weights. Below, we will address these uniformly under the framework of \eqref{eq:Nparticlesys}.

\section{The strong convergence}\label{sec:strong}

Since there is no exchangeability, we have to use weighted average of the errors. 
We consider
\begin{gather}\label{eq:strongerror}
J(t):=\frac{1}{2N}\sum_{i=1}^N m_i \E|\tilde{X}^i-X^i|^2,
\end{gather}
which is the strong error. As a common convention, the ``$\frac{1}{2}$'' prefactor is used for energies of quadratic forms. Recently, a certain quantum correspondence of this error has been used by Golse et al. to prove the convergence of Random Batch Method for $N$-body Schr\"odinger equations \cite{golse2019random}.

\begin{theorem}\label{thm:strongconv}
Suppose Assumptions \ref{ass:mass}--\ref{ass:kernelfunctions} hold. Then, it holds that
\begin{gather}
\sup_{t\le T}J(t)\le C\sqrt{\frac{\tau}{p-1}+\tau^2},
\end{gather}
where $C$ is independent of $N, p$. 
\end{theorem}

Define the error of the random approximation for the interacting force by
\begin{gather}\label{eq:errorofforce}
\begin{split}
    \chi_i(x)&:=\frac{1}{p-1}\sum_{j\in \mathcal{C}_{\theta},j\neq i }m_jK_{ij}(x_i, x_j)
    -\frac{1}{N-1}\sum_{j: j\neq i}m_j K_{ij}(x_i, x_j)\\
    &=\frac{1}{p-1}\sum_{j: j\neq i }I_{ij}m_jK_{ij}(x_i, x_j)
    -\frac{1}{N-1}\sum_{j: j\neq i}m_j K_{ij}(x_i, x_j),
\end{split}
\end{gather}
where $x_i\in \R^d$, and again $x=(x_1, \cdots, x_N)\in \R^{Nd}$.

We have the following facts 
\begin{lemma}\label{lmm:estimatesofindicator}
For $i\neq j$, it holds that
\begin{gather}
\mathbb{E}I_{ij}=\frac{p-1}{N-1},
\end{gather}
and for distinct $i,j,\ell$, it holds that
\begin{gather}
\E I_{ij}I_{i\ell}=\frac{(p-1)(p-2)}{(N-1)(N-2)}.
\end{gather}
\end{lemma}
The proofs have been done in the proof of \cite[Lemma 3.1]{jin2020random}, for which we omit. Using Lemma \ref{lmm:estimatesofindicator}, one can have the following  consistency of the Random Batch Method.
\begin{lemma}\label{thm:consistency}
    For given $x=(x_1, \ldots, x_N)\in \mathbb{R}^{Nd}$, it holds that
    \begin{gather}
        \mathbb{E}\chi_i(x)=0.
    \end{gather}
    Moreover, the second moment is given by
\begin{gather}\label{eq:var1}
        \mathbb{E}|\chi_i(x)|^2 = 
        \left(\frac{1}{p-1}-\frac{1}{N-1}\right)\Lambda_i(x),
\end{gather}
    where
\begin{gather}\label{eq:Lambda}
        \Lambda_i(x):=\frac{1}{N-2}
        \sum_{j: j\neq i}\Big| m_j K_{ij}(x_i, x_j)-\frac{1}{N-1}
        \sum_{\ell: \ell\neq i}m_\ell K_{i\ell}(x_i, x_\ell)  \Big|^2.
\end{gather}
\end{lemma}
Though slightly different from \cite[Lemma 3.1]{jin2020random}, the proof is essentially the same and we omit.

We move to some important additional estimates:
\begin{lemma}\label{lmm:conditionalest}
Let $X^i$ and $\tilde{X}^i$ be solutions to \eqref{eq:Nparticlesys} and 
\eqref{eq:firstalgorithm} respectively. Suppose Assumptions \ref{ass:mass}--\ref{ass:kernelfunctions} hold. Then,
\begin{gather}\label{eq:uniformmoments}
\sup_{t\le T}(\mathbb{E}|X^i(t)|^q+\mathbb{E}|\tilde{X}^i(t)|^q)\le C_q.
\end{gather}
Besides, for any $k>0$ and $q\ge 2$,
\begin{gather}\label{eq:conditionmoment}
\sup_{t\in [t_{k-1}, t_k)}\left|\mathbb{E}(|\tilde{X}^i(t)|^q |\mathcal{F}_{k-1})\right|
\le C_1|\tilde{X}^i(t_{k-1})|^q+C_2.
\end{gather}
holds almost surely.

Moreover, almost surely, it holds that
\begin{gather}\label{eq:conditionalincrements}
\begin{split}
& |\E(\tilde{X}^i(t)-\tilde{X}^i(t_{k-1}) | \mathcal{F}_{k-1}) |
\le C(1+|\tilde{X}^i(t_{k-1})|^q)\tau,\\
& \Big|\mathbb{E}\left(|\tilde{X}^i(t)-\tilde{X}^i(t_{k-1})|^2 | \mathcal{F}_{k-1}\right) \Big| \le C(1+|\tilde{X}^i(t_{k-1})|^{q}) \tau.
\end{split}
\end{gather}
\end{lemma}
Note that the second equation in \eqref{eq:conditionalincrements} is different from that in \cite{jin2020random}. In fact, the proof in  \cite{jin2020random} has a small gap, and one needs this refined estimate to fill in that gap as well (\cite{jin2020random} Page 13, from Line 17 to Line 19, the variable $\E(|\delta \tilde{X}^i|^2+|\delta\tilde{X}^j|^2 | \mathcal{F}_{m-1})$ is not independent of $\mathcal{C}_{\theta}$; to get Line 19, we need this refined version). Though the proof is not hard, we attach it in Appendix \ref{app:conditionalmoments}.

The following lemma is an improved version of \cite[Lemma 3.2]{jin2020random}, which is very important to establish the strong convergence for the problem considered in this paper.
\begin{lemma}\label{lmm:normofrandomsum}
Fix $i\in\{1,\ldots, N\}$. Let $\mathcal{C}_{\theta}$ be the random batch of size $p$ that contains $i$ in the random division. 
Let $Y_j$ ($1\le j\le N$) be $N$ random variables (or random vectors) that are independent of $\mathcal{C}_{\theta}$. Then, for $p\ge 2$,
\begin{gather}
\left\|\frac{1}{p-1}\sum_{j\in\mathcal{C}_{\theta},j\neq i}Y_j\right\|\le 
\left(\frac{1}{N-1}\sum_{j: j\neq i}\|Y_j\|^2\right)^{1/2}.
\end{gather}
\end{lemma}

\begin{proof}
By the definition and independence,
\[
\begin{split}
\left\|\frac{1}{p-1}\sum_{j\in\mathcal{C}_{\theta},j\neq i}Y_j\right\|^2
&=\frac{1}{(p-1)^2}\mathbb{E}\left|\sum_{j: j\neq i}I_{ij}Y_j\right|^2 \\
&=\frac{1}{(p-1)^2}\sum_{j,\ell:j\neq i,\ell\neq i}\mathbb{E}(I_{ij}I_{i\ell})\mathbb{E}(Y_j\cdot Y_\ell)\\
&\le \frac{1}{(p-1)^2}\Bigg[\sum_{j,\ell:j\neq i,\ell\neq i,j\neq\ell}\frac{(p-1)(p-2)}{(N-1)(N-2)}\|Y_j\| \|Y_\ell\|  \\
&\quad\quad\quad\quad\quad\quad+\sum_{j:j\neq i}\E I_{ij} \|Y_j\|^2\Bigg]\\
&=:R_1+R_2.
\end{split}
\]
 Note that the independence is used in the second equality.
 The first inequality is due to Lemma \ref{lmm:estimatesofindicator}.
 
 It is easy to calculate
 \[
 \begin{split}
 R_1&\le \frac{p-2}{(p-1)(N-1)(N-2)}\sum_{j,\ell:j\neq i,\ell\neq i,j\neq\ell}(\frac{1}{2}\|Y_j\|^2+\frac{1}{2}\|Y_{\ell}\|^2)\\
&=\frac{p-2}{(p-1)(N-1)}\sum_{j:j\neq i}\|Y_j\|^2,
\end{split}
 \]
 while
 \[
 R_2=\frac{1}{(p-1)(N-1)}\sum_{j:j\neq i}\|Y_j\|^2.
 \]
 Hence,
 \[
 R_1+R_2\le \frac{1}{N-1}\sum_{j: j\neq i}\|Y_j\|^2.
 \]
 The claim thus follows.
\end{proof}

We now consider the error process defined in \eqref{eq:Zidef}.
The derivative of $Z^i$ is clearly given by
\begin{multline}\label{eq:Zderi}
\frac{d}{dt}Z^i=[b(\tilde{X}^i)-b(X^i)]
+\frac{1}{p-1}\sum_{j\in \mathcal{C}_{\theta},j\neq i}m_j K_{ij}(\tilde{X}^i, \tilde{X}^j)\\
-\frac{1}{N-1}\sum_{j\in \mathcal{C}_{\theta}, j\neq i}m_j K_{ij}(X^i, X^j),
\end{multline}
where $\mathcal{C}_{\theta}$ is the random batch in $\mathcal{C}$
that contains $i$.

Define
\begin{gather}\label{eq:deltaKnewdef}
\delta K_{ij}(t):=K_{ij}(\tilde{X}^i(t),\tilde{X}^j(t))-K_{ij}(X^i(t),X^j(t)).
\end{gather}
The right hand side of  \eqref{eq:Zderi} can then be written as 
\begin{gather}\label{eq:Zderi}
\begin{split}
\frac{d}{dt}Z^i
&=[b(\tilde{X}^i)-b(X^i)]+\frac{1}{N-1}\sum_{j: j\neq i}m_j \delta K_{ij}
+\chi_i(\tilde{X})\\
&=[b(\tilde{X}^i)-b(X^i)]+\frac{1}{p-1}\sum_{j\in \mathcal{C}_{\theta}, j\neq i}m_j \delta K_{ij}+\chi_i(X) .
\end{split}
\end{gather}

With this, one can obtain the following simple lemma
\begin{lemma}\label{lmm:Zestimate}
For $t\in [t_{k-1}, t_k)$, 
\begin{gather}
\|Z^i(t)-Z^i(t_{k-1})\|\le C \tau.
\end{gather}
Also, almost surely,
\begin{gather}
|Z^i(t)|\le |Z^i(t_{k-1})|(1+C\tau)+C\tau.
\end{gather}
\end{lemma}

\begin{proof}
By \eqref{eq:Zderi}, since $b$ has polynomial growth, the claim for $\|Z^i(t)-Z^i(t_{m-1})\|$ is then an easy consequence of the $q$-moment estimates in Lemma \ref{lmm:conditionalest}.

Dotting  using \eqref{eq:Zderi} with $Z^i$ and using the one-sided Lipschitz condition in Assumption \ref{ass:kernelfunctions}, one has 
\[
\frac{1}{2}\frac{d}{dt}|Z^i|^2\le C|Z^i|^2+C|Z^i|.
\]
Hence, almost surely, it holds that
\[
\frac{d}{dt}|Z^i|\le C|Z^i|+C.
\]
The second claim then follows.
\end{proof}

We now give the proof of the strong convergence in Theorem \ref{thm:strongconv}.
\begin{proof}[Proof of Theorem \ref{thm:strongconv}]
First,
\[
\frac{dJ(t)}{dt}=\frac{1}{N}\sum_{i=1}^N m_i \E \Big\{ Z^i\cdot [(F_i(\tilde{X})-F_i(X))
+\chi_i(\tilde{X})]\Big\}.
\]

The first term is easy to bound by the one-sided Lipschitz condition of $b$ and the conditions of $K_{ij}$ in Assumption \ref{ass:kernelfunctions} :
\[
\frac{1}{N}\sum_{i=1}^N m_i \E \Big\{Z^i\cdot (F_i(\tilde{X})-F_i(X))\Big\}
\le \frac{C}{N}\sum_i m_i |Z^i|^2+\frac{C}{N(N-1)}\sum_{i\neq j} m_i m_j \E(|Z^i|^2+|Z^j||Z^j|).
\]
This is clearly bounded by $J(t)$.

Now, we focus on the second term. The technique is the same as in our previous work \cite{jin2020random}, but some special modifications are needed for our problem here.

\begin{gather*}
\begin{split}
\frac{1}{N}\sum_{i=1}^N m_i \E \Big\{Z^i(t)\cdot \chi_i(\tilde{X}(t))\Big\}
=&\frac{1}{N}\sum_{i=1}^N m_i \E \Big\{Z^i(t_{k-1})\cdot \chi_i(\tilde{X}(t))\Big\}\\
&+\frac{1}{N}\sum_{i=1}^N m_i \E \Big\{(Z^i(t)-Z^i(t_{k-1}))\cdot \chi_i(\tilde{X}(t))\Big\}\\
=:&I_1+I_2.
\end{split}
\end{gather*}

{\bf Step 1--}  Estimate of $I_1$:

For $I_1$, using the consistency result in Lemma \ref{thm:consistency}, one has
\[
I_1=\frac{1}{N}\sum_{i=1}^N m_i \E \left\{Z^i(t_{k-1})\cdot 
\Big[\chi_i(\tilde{X}(t)) - \chi_i(\tilde{X}(t_{k-1})) \Big]\right\}.
\]
In fact,
\[
\E \left\{Z^i(t_{k-1})\cdot 
\chi_i(\tilde{X}(t_{k-1}))\right\}
=\E\left\{ Z^i(t_{k-1})\cdot \E\Big[  
\chi_i(\tilde{X}(t_{k-1}))|\mathcal{G}_{k-1}\Big]\right\}
=\E 0=0.
\]
This is the only place where the $\sigma$-algebra $\mathcal{G}_{k-1}$
is used.

Note that
\[
\begin{split}
&\E \left\{Z^i(t_{k-1})\cdot 
\Big[\chi_i(\tilde{X}(t)) - \chi_i(\tilde{X}(t_{k-1})) \Big]\right\}\\
&=\mathbb{E}\Bigg( Z^1(t_{k-1})\cdot \mathbb{E}(\chi_{i}(\tilde{X}(t))-\chi_{i}(\tilde{X}(t_{k-1})) |\mathcal{F}_{k-1}) \Bigg) \\
& \le C\|Z^1(t_{k-1})\| \left\|\mathbb{E}[\chi_{i}(\tilde{X}(t))-\chi_{i}(\tilde{X}(t_{k-1})) 
|\mathcal{F}_{k-1}] \right\|.
\end{split}
\]

Using the definition of $\chi_i$ \eqref{eq:errorofforce}, one has
\[
\begin{split}
&\mathbb{E}\left[\chi_{i}(\tilde{X}(s))-\chi_{i}(\tilde{X}(t_{k-1})) |\mathcal{F}_{k-1} \right] \\
&=\frac{1}{p-1}\sum_{j\in \mathcal{C}_{\theta},j\neq i} m_j\mathbb{E}(\delta \tilde{K}^{ij}|\mathcal{F}_{k-1})
-\frac{1}{N-1}\sum_{j:j\neq i}m_j\mathbb{E}(\delta\tilde{K}^{ij}|\mathcal{F}_{k-1}),
\end{split}
\]
where 
\[
\delta\tilde{K}^{ij}=K_{ij}(\tilde{X}^i(s),\tilde{X}^j(s))-K_{ij}(\tilde{X}^i(t_{k-1}), \tilde{X}^j(t_{k-1})).
\]

We first estimate $\mathbb{E}(\delta \tilde{K}^{ij}|\mathcal{F}_{k-1})$.
Denote $\delta\tilde{X}^j:=\tilde{X}^j(s)-\tilde{X}(t_{k-1})$. Performing Taylor expansion around $t_{k-1}$, one has
\[
\delta\tilde{K}^{ij}=(\nabla_{x_i} K_{ij}|_{t_{k-1}}\cdot\delta\tilde{X}^i+
\nabla_{x_j}K_{ij}|_{t_{k-1}}\cdot\delta\tilde{X}^j)
+\frac{1}{2}M:[\delta\tilde{X}^i,\delta\tilde{X}^j]\otimes[\delta\tilde{X}^i,\delta\tilde{X}^j],
\]
with $M$ being a random variable (tensor) bounded by $\|\nabla^2K\|_{\infty}$.
By \eqref{eq:conditionalincrements}, one finds that
 \[
 |\mathbb{E}(\delta \tilde{K}^{ij}|\mathcal{F}_{k-1})|
 \le C(1+|\tilde{X}^i(t_{k-1})|^{q}+|\tilde{X}^j(t_{k-1})|^{q})\tau.
 \]
The right hand side is independent of $\mathcal{C}_{\theta}$, and this is the place where we need the almost surely bound \eqref{eq:conditionalincrements}. Applying Lemma \ref{lmm:normofrandomsum}, one has
\[
\left\|\frac{1}{p-1}\sum_{j\in \mathcal{C}_{\theta},j\neq i} m_j\mathbb{E}(\delta \tilde{K}^{ij}|\mathcal{F}_{k-1}) \right\|
\le C\tau \left(1+\left(\frac{1}{N-1}\sum_{j: j\neq i}\||\tilde{X}^j(t_{k-1})|^{q_1}\|^2\right)^{1/2}\right)\le C\tau.
\]
The term $\|\frac{1}{N-1}\sum_{j:j\neq i}m_j\mathbb{E}(\delta\tilde{K}^{ij}|\mathcal{F}_{k-1})\|$ is much easier to estimate, and it is also bounded by $C\tau$.

Hence
\[
\E Z^i(t_{k-1})\cdot 
\Big[\chi_i(\tilde{X}(t)) - \chi_i(\tilde{X}(t_{k-1})) \Big]
\le C\|Z^i(t_{k-1})\|\tau\le C\|Z^i(t)\|\tau+C\tau^2.
\]

Then,
\[
I_1\le \left(\frac{C}{N}\sum_{i=1}^N m_i \|Z^i(t)\|\tau \right)+C\tau^2 \le 
\delta \frac{1}{N}\sum_{i=1}^N m_i \|Z^i(t)\|^2+C(\delta)\tau^2
\]

{\bf Step 2--} Estimate of $I_2$.

We decompose
\begin{multline*}
I_2=\frac{1}{N}\sum_{i=1}^N m_i \E (Z^i(t)-Z^i(t_{k-1}))\cdot [\chi_i(\tilde{X}(t))-\chi_i(X(t))]\\
+\frac{1}{N}\sum_{i=1}^N m_i \E (Z^i(t)-Z^i(t_{k-1}))\cdot \chi_i(X(t))
=:I_{21}+I_{22}.
\end{multline*}

We first consider $I_{21}$.Clearly,
\[
I_{21}\le \frac{1}{N}\sum_{i=1}^N m_i C\tau \|\chi_i(\tilde{X}(t))-\chi_i(X(t))\|.
\]
Since $K$ is Lipschitz continuous,
\[
|\delta K_{ij}(t)|\le L(|Z^i(t)|+|Z^j(t)|).
\]
Hence,
\begin{gather}\label{eq:randomsumaux}
\left\|\left(\frac{1}{p-1}\sum_{j\in\mathcal{C}_{\theta},j\neq i} m_j \delta K_{ij}(t)\right)\right\|
\le L\left(\|Z^i(t)\|+\left\|  \frac{1}{p-1}\sum_{j\in\mathcal{C}_{\theta},j\neq i}m_j |Z^j(t)|\right\|\right).
\end{gather}
Note that $Z^j(t)$ depends on $\mathcal{C}_{\theta}$ and we cannot apply Lemma \ref{lmm:normofrandomsum}. 
Instead, by Lemma \ref{lmm:Zestimate}, one has that
\[
|Z^i(t)|\le C|Z^i(t_{k-1})|+C\tau.
\]
Since $Z^i(t_{k-1})$ is independent of $\mathcal{C}_{\theta}$, Lemma \ref{lmm:normofrandomsum} then gives us that
 \begin{gather}\label{eq:proofaux2}
 \begin{split}
\left\|  \frac{1}{p-1}\sum_{j\in\mathcal{C}_{\theta},j\neq i}m_j|Z^j(t)|\right\| & \le C\left(\frac{1}{N-1}\sum_{j: j\neq i}m_j^2 \|Z^j(t_{k-1})\|^2\right)^{1/2}+C\tau \\
& \le C\left(\frac{1}{N-1}\sum_{j: j\neq i}m_j \|Z^j(t_{k-1})\|^2\right)^{1/2}+C\tau.
 \end{split}
 \end{gather}
\begin{remark}
This is the place where we need Lemma \ref{lmm:normofrandomsum}, a refined version of \cite[Lemma 3.2]{jin2020random}.
If we use \cite[Lemma 3.2]{jin2020random}, one controls $\|  \frac{1}{p-1}\sum_{j\in\mathcal{C}_{\theta},j\neq i}m_j|Z^j(t_{k-1})|\|$ by $\sup_j m_j \|Z^j(t_{k-1})\|$ which is not enough to close the estimate.
\end{remark}

With this, we find
\[
\begin{split}
I_{21} &\le C\left(\frac{1}{N}\sum_{i=1}^N m_i \tau \left(\frac{1}{N-1}\sum_{j: j\neq i}m_j \|Z^j(t_{k-1})\|^2\right)^{1/2}\right)+C\tau^2\\
& \le \delta\frac{1}{N}\sum_i m_i\|Z^i(t_{k-1})\|^2+C(\delta)\tau^2,
\end{split}
\]
where Assumption \ref{ass:mass} has been used.

We now consider $I_{22}$. We first recall
\begin{gather}\label{eq:Zderipf}
\frac{d}{dt}Z^i=[b(\tilde{X}^i)-b(X^i)]+\frac{1}{p-1}\sum_{j\in\mathcal{C}_{\theta},j\neq i}
m_j \delta K_{ij}(t)+\chi_i(X(t))
\end{gather}
Note
\[
|b(\tilde{X}^i)-b(X^i)|\le \int_0^1|(\tilde{X}^i-X^i)\cdot \nabla b((1-\sigma)\tilde{X}^i+\sigma X^i)|\,d\sigma,
\]
and $\nabla b((1-\sigma)\tilde{X}^i+\sigma X^i)$ is controlled by $C(|\tilde{X}^i|^{q}+|X^i|^{q})$ for some $q>0$. Hence,
\begin{gather*}
\begin{split}
\mathbb{E}|b(\tilde{X}^i)-b(X^i)||\chi_{i}(X(t'))|
& \le \|\chi_{i}(X(t'))\|_{\infty}\|\tilde{X}^i-X^i\| (\mathbb{E}(|\tilde{X}^i|^{q_1}+|X^i|^{q_1})^2)^{1/2} \\
&\le C\|Z^i(t)\|,
\end{split}
\end{gather*}
by \eqref{eq:uniformmoments}.

Integrating \eqref{eq:Zderipf} in time over $[t_{k-1}, t]$, then dotting with $\chi_{i}(X(t))$, and taking the expectation,  one gets  
\begin{multline}\label{eq:proofaux1}
\left| \mathbb{E}((Z^i(t)-Z^i(t_{k-1}))\cdot\chi_{m,i}(X(t))) \right|
\le C\int_{t_{k-1}}^t\|Z^i(s)\| ds+\\
\int_{t_{k-1}}^{t}\mathbb{E}\left[\left(\frac{1}{p-1}\sum_{j\in\mathcal{C}_{\theta},j\neq i}m_j \delta K_{ij}(s)\right)\cdot\chi_{i}(X(t))\right]\,ds
+\mathbb{E}\int_{t_{k-1}}^{t}\chi_{i}(X(s))\cdot\chi_{i}(X(t))\,ds,
\end{multline}
Applying \eqref{eq:randomsumaux} and \eqref{eq:proofaux2}, the second term on the right hand of \eqref{eq:proofaux1} is bounded by 
\[
C \left(\frac{1}{N-1}\sum_{j: j\neq i}m_j^2 \|Z^j(t_{k-1})\|^2\right)^{1/2}\tau+C\tau^2.
\]
Similarly as we estimate $I_{21}$, this is controlled by $\delta\frac{1}{N}\sum_i m_i\|Z^i(t_{k-1})\|^2+C(\delta)\tau^2$.
The last term on the right hand of \eqref{eq:proofaux1} is controlled by Lemma \ref{thm:consistency} (in particular, equation \eqref{eq:var1}):
\[
\left[\frac{1}{p-1}-\frac{1}{N-1} \right] \|\Lambda_i\|_{\infty} \tau.
\]
Therefore,
\[
I_{22}\le \delta\frac{1}{N}\sum_i m_i\|Z^i(t_{k-1})\|^2+C(\delta)\tau^2
+\left[\frac{1}{p-1}-\frac{1}{N-1} \right] \|\Lambda_i\|_{\infty} \tau.
\]
Moreover,
\[
\|Z^i(t_{k-1})\|^2\le (\|Z^i(t)\|+C\tau)^2\le 2\|Z^i(t)\|^2+2C^2\tau^2.
\]

Finally, taking all those estimates together, one has the following estimate:
\[
\frac{d}{dt}J\le CJ+C(\delta)\tau^2
+\left[\frac{1}{p-1}-\frac{1}{N-1} \right] \|\Lambda_i\|_{\infty} \tau.
\]
The claim then follows by Gr\"onwall's lemma.
\end{proof}

\begin{remark}
The strong convergence can imply the convergence of one marginal distributions. See \cite{jin2020random} for more details. 
The weak convergence below, however, is for empirical measure \eqref{eq:empiricalmeasure}. 
\end{remark}

\section{The weak convergence and Random Batch Method for backward equations}\label{sec:weak}

In practice, one may be more interested in the distributions of the particles
 instead of the trajectories of $X^i(t)$. Hence, the error \eqref{eq:strongerror} is not suitable for this purpose, and we seek to study the weak convergence of the empirical measure \eqref{eq:empiricalmeasure}, as commonly used in the numerical SDE literature \cite{milstein2013stochastic,kloeden2013numerical}.
Roughly speaking, we say a sequence of measures $\mu^N$ converges to some measure $\mu$ weakly if for any suitable test function $\varphi$, it holds that
\[
\int \varphi\, d\mu^N=:\langle \mu^N, \varphi\rangle\to \langle \mu, \varphi\rangle:=\int \varphi\, d\mu.
\]
 
For our problem, we consider the weak convergence of the empirical measure $\mu^N$ defined in \eqref{eq:empiricalmeasure}, which is defined on Borel sets from $\R^d$, to the empirical measure corresponding to \eqref{eq:Nparticlesys} as $\tau\to 0$. Hence, to show that the empirical measures given by $X$ and $\tilde{X}$ are close in law, we pick a test function $\varphi\in C_b^{\infty}(\R^d)$, and hope to show that the weak error defined below is small:
\begin{gather}\label{eq:weakerror}
E_k:=\left|\frac{1}{N}\sum_{i=1}^N \omega_i \E \varphi(\tilde{X}^i(k\tau))
-\frac{1}{N}\sum_{i=1}^N \omega_i \E \varphi(X^i(k\tau)) \right|.
\end{gather} 

\begin{remark}
Traditionally, the test functions used for schemes of numerical
SDEs are those with polynomial growth at infinity \cite{milstein86,milstein2013stochastic}.
We used $C_b^{\infty}$ as the test functions as done nowadays \cite{am11,li2018numerical}, which will induce a weaker topology that disregards high order moments. If one has the moment control, the convergence using these two types of test functions will be the same.
\end{remark}

For the weak convergence, we need some different assumptions on $b$
and $K_{ij}$. 
\begin{assumption}\label{eq:boundedderivatives}
The functions $b,K_{ij}$ are $C^4$ and the derivatives of $b$ and $K_{ij}$ up to order $4$ are uniformly bounded
(uniform in $i,j$).
\end{assumption}

\begin{remark}
Proof of weak convergence using Assumption \ref{ass:kernelfunctions} instead of \ref{eq:boundedderivatives} should also be completed.
Using Assumption \ref{eq:boundedderivatives} makes the proof based on semigroup technique elegant (see the details below).
\end{remark}

We now state the main theorem for the weak convergence of RBM.
\begin{theorem}\label{thm:weakconvergence}
Under Assumption \ref{ass:mass} and Assumption \ref{eq:boundedderivatives}, the Random Batch Method converges weakly with first order for the empirical measure \eqref{eq:empiricalmeasure}. In particular, the weak error defined in \eqref{eq:weakerror} satisfies
\begin{gather}
\sup_{k: k\tau\le T}|E_k|\le C\tau,
\end{gather}
where $C=C(\varphi, T)$ is independent of $N, \tau$, but depends on $\varphi$.
\end{theorem}

Usually, in numerical SDEs, to prove the weak convergence, one makes use of the backward Kolmogorov equation ("backward equation" for short)
which is defined in the same Euclidean space. For our problem, we need to lift the Euclidean space from $\R^d$ to $\R^{Nd}$. In particular, define
\begin{gather}\label{eq:u}
u(x, t):= \frac{1}{N}\sum_{i=1}^N \omega_i \E [\varphi(X^i(t)) | X(0)=x]
\end{gather}
where $x\in \R^{Nd}$ and $X$ is the solution to  \eqref{eq:Nparticlesys}.
Then, we make use of this function to study the weak convergence. This function $u$ satisfies the following backward equation \cite{oksendal03}
\begin{gather}\label{eq:backward}
\partial_t u=\cL u: =\sum_{i=1}^N \left[b(x_i)+\frac{1}{N-1}\sum_{j: j\neq i} m_j K_{ij}(x_i, x_j)\right]\cdot\nabla_{x_i}u+\sum_{i=1}^N \frac{1}{2}\sigma^2\Delta_{x_i}u.
\end{gather}
The operator $\mathcal{L}$ is called the generator of the ODE/SDE \eqref{eq:Nparticlesys}. The Laplacian $\Delta_{x_i}$ is given by
\begin{gather}
\Delta_{x_i}:=\sum_{j=1}^d \partial_{x_i^{(j)}x_i^{(j)}},
~~x_i=(x_i^{(1)},\cdots, x_i^{(d)})\in \R^d.
\end{gather}
The solution semigroup for \eqref{eq:backward} will be denoted by
\begin{gather}
e^{t\mathcal{L}}u(\cdot, 0):=u(x,t).
\end{gather}
By the well-known property of backward equation \cite{oksendal03,li2020large}, one has
\begin{gather}\label{eq:semigroup}
\|u(\cdot, t)\|_{\infty}=\|e^{t\mathcal{L}}u(\cdot, 0)\|_{\infty}\le \|u(\cdot, 0)\|_{\infty}.
\end{gather}

The function $u$ is defined on $\R^{Nd}$, and naive estimates of the norms for the derivatives will depend on $N$. This is not sufficient for us to show the $N$ independent of weak convergence for the empirical measure.  In fact, it is clear that
\begin{gather}
\|u\|_{\infty}\le C,
\end{gather}
independent of $N$. The following proposition then gives the crucial estimates of the derivatives of $u$.
\begin{proposition}\label{pro:boundofderivatives}
For any $i$ and $j\neq i$, one has
\begin{gather}
\|\nabla_{x_i} u\|_{\infty}+\|\nabla^2_{x_i} u\|_{\infty}+\|\nabla_{x_i}^3u\|_{\infty}+\|\nabla_{x_i}^4u\|_{\infty}\le C(T)\frac{1}{N},
\end{gather}
and
\begin{gather}
\|\nabla_{x_i}\nabla_{x_j}u\|_{\infty}
+\|\nabla_{x_i}\nabla^2_{x_j}u\|_{\infty}+\|\nabla^2_{x_i}\nabla^2_{x_j}u\|_{\infty} \le C(T)\frac{1}{N^2}.
\end{gather}
\end{proposition}
Here, $\nabla_{x_i}^2$ is the Hessian matrix. Similarly, 
$\nabla_{x_i}^4u$ is the fourth order tensor with derivatives 
of the form $(\prod_{k=1}^4\partial_{x_i^{(j_k)}})u$ and $j_k\in\{1,\cdots, d\}$. The norm $\|\nabla_{x_i}^4u\|_{\infty}$
is understood as
\[
\sup_{x}\left[\sum_{j_1, j_2, j_3, j_4}\Big(\big(\prod_{k=1}^4\partial_{x_i^{(j_k)}}\big)u\Big)^2\right]^{1/2}.
\]
The proof is kind of tedious and is given in Appendix \ref{app:uderi}.

For further discussion, we introduce the generator corresponding to the 
Random Batch Method. For $t\in (t_{k-1}, t_k]$
\begin{gather*}
\cL_C:= \sum_{i=1}^N \left[b(x_i)+\frac{1}{p-1}\sum_{j: j\neq i} I_{ij}^{(k-1)} m_j K_{ij}(x_i, x_j)\right]\cdot\nabla_{x_i}
+\sum_{i=1}^N \frac{1}{2}\sigma^2\Delta_{x_i}.
\end{gather*}
We recall that $I_{ij}$ is the indicator for $i, j$ being in the same batch,
and we use $I_{ij}^{(k-1)}$ to mean the indicator corresponding to $\mathcal{C}^{(k-1)}$.

The importance of Proposition \ref{pro:boundofderivatives} is that one can bound $\cL^i u$ uniformly in $N$, where $\cL^i$ means the composition of $\cL$ for $i$ times (if $i=0$, it is the identity operator). This is crucial for establishing the weak convergence uniformly in $N$. In particular, we have the following.
\begin{lemma}\label{lmm:generatorofu}
For the function $u$ defined in \eqref{eq:u}, it holds that for $i=0,1,2$ that
\begin{gather}
\|\mathcal{L}^iu\|\le C,~~~\|\mathcal{L}_{\mathcal{C}}^iu\|\le C
\end{gather}
Moreover,
\begin{gather}
\E\mathcal{L}_{\mathcal{C}}=\mathcal{L}.
\end{gather}
\end{lemma}
\begin{proof}
The first assertion is a corollary of Proposition \ref{pro:boundofderivatives}.
The second claim is the one proved in Lemma \ref{thm:consistency}.
\end{proof}

To go further, we need to introduce a semigroup associated to RBM (see also \cite{feng2018} for the semigroup used in SGD).
Consider the transition $\tilde{X}(t_{k-1})\to \tilde{X}(t_k)$, it is clear that this transition gives a time homogeneous Markov chain. Define the operator $\mathcal{S}^{(k)}: C_b(\R^{Nd})\to C_b(\R^{Nd})$ as
\begin{gather}
\mathcal{S}^{(k)}\phi:=\E [\phi(\tilde{X}(k\tau))| \tilde{X}(0)=x].
\end{gather}
Below, we sometimes use $\E_x$ to mean $\E (\cdot |\tilde{X}(0)=x)$ for convenience. Then, by the Markov property \cite{durrett2010}, it can be shown that
\begin{gather}\label{eq:Ssemigroup}
\mathcal{S}^{(k)}=\mathcal{S}^k:=\mathcal{S}\circ  \cdots \circ \mathcal{S}.
\end{gather}
In fact, for any test function $\phi$, we let $u^k:=(\mathcal{S}^{(k)}\phi$. Then, it holds that
\[
\begin{split}
u^k(x)&=\E[\phi(\tilde{X}(t_k)) | \tilde{X}(0)=x]\\
&=\E[ \E[\phi(\tilde{X}(t_k)) | \tilde{X}(t_1)] | \tilde{X}(0)=x ] \\
&=\E[ u^{k-1}(\tilde{X}(t_1)) |  \tilde{X}(0)=x],
\end{split}
\]
where the third equality is by the Markov property, namely
\[
\mathcal{S}^{(k)}=\mathcal{S}\circ\mathcal{S}^{(k-1)}.
\]
Clearly, $\mathcal{S}^k$ are nonexpansive in $L^{\infty}(\R^{Nd})$, i.e.
\begin{gather}\label{eq:linfitynonexpansive}
 \| \mathcal{S}^k \phi\|_{\infty}\le \|\phi\|_{\infty}.
\end{gather}

We give the proof of the weak convergence.
\begin{proof}[Proof of Theorem \ref{thm:weakconvergence}]

The weak error in \eqref{eq:weakerror} is to compare $u$ in \eqref{eq:u} and $\mathcal{S}^n\phi$ for test function
\begin{gather}
\phi(x)=\frac{1}{N}\sum_{i=1}^N \omega_i \varphi(x_i).
\end{gather}

By Equation \eqref{eq:Ssemigroup} ($\mathcal{S}^{(k)}=\mathcal{S}^k$), we need to estimate
\[
\mathcal{S}^n\phi(x)-u(x, n\tau)=\sum_{k=1}^n \Big[ \mathcal{S}^k u(x, (n-k)\tau)-\mathcal{S}^{k-1} u(x, (n-k+1)\tau) \Big].
\]
Clearly, by \eqref{eq:linfitynonexpansive}, one has
\[
|\mathcal{S}^n\phi(x)-u(x, n\tau)|\le \sum_{k=1}^n \|\mathcal{S}u(x, (n-k)\tau)-u(x, (n-k+1)\tau)\|_{\infty}
\]

By the definition of $S$, one has
\[
\begin{split}
\mathcal{S}u(x, (n-k)\tau) &=\E_x u(\tilde{X}(\tau), (n-k)\tau ) \\
&=\E_{\mathcal{C}} \E_x( u(\tilde{X}(\tau), (n-k)\tau) | \mathcal{C}) \\
&=\E_{\mathcal{C}} e^{\tau\mathcal{L}_{\mathcal{C}}}u(x, (n-k)\tau ).
\end{split}
\]
The second equality means that we fix a division of batches to compute the expectation with respect to the Brownian motions, and then average out about the random batches. For the last equality, we recall that $e^{\tau\mathcal{L}_{\mathcal{C}}}$ means the solution semigroup by $\partial_t v=\mathcal{L}_{\mathcal{C}}v$. 

Note 
\begin{multline*}
e^{\tau\mathcal{L}_{\mathcal{C}}}u(x, (n-k)\tau )
=u(x, (n-k)\tau)+\tau\mathcal{L}_{\mathcal{C}}u(x, (n-k)\tau)\\
+\int_0^{\tau}(\tau-z)\mathcal{L}_{\mathcal{C}}^2e^{z\mathcal{L}_{\mathcal{C}}}u(x, (n-k)\tau)\,dz.
\end{multline*}
For the remainder term, by Lemma \ref{lmm:generatorofu}, it holds for every partition of batches that 
\[
\begin{split}
\|\mathcal{L}_{\mathcal{C}}^2e^{z\mathcal{L}_{\mathcal{C}}}u(x, (n-k)\tau)\|_{\infty}
&=\|e^{z\mathcal{L}_{\mathcal{C}}}\mathcal{L}_{\mathcal{C}}^2u(x, (n-k)\tau)\|_{\infty} \\
& \le \|\mathcal{L}_{\mathcal{C}}^2u(x, (n-k)\tau)\|_{\infty}\le C.
\end{split}
\]
Hence,
\[
\|\mathcal{S}u(x, (n-k)\tau)-u(x, (n-k)\tau)-\E\tau\mathcal{L}_{\mathcal{C}}u(x, (n-k)\tau)\|_{\infty}\le C\tau^2,
\]
with $C$ independent of $N$. Here, we used the fact that $e^{\tau\mathcal{L}_{\mathcal{C}}}$ is nonexpansive in $L^{\infty}$, similar as in \eqref{eq:semigroup}.

Similarly,
\[
\begin{split}
&u(x, (n-k+1)\tau) =e^{\tau \mathcal{L}}u(x, (n-k)\tau)\\
&=u(x, (n-k)\tau)+\tau\mathcal{L}u(x, (n-k)\tau)
+\int_0^{\tau}(\tau-z)\mathcal{L}^2e^{z\mathcal{L}}u(x, (n-k)\tau)\,dz.
\end{split}
\]
The remainder term is again bounded by $C\tau^2$.
Since
\[
\E\mathcal{L}_{\mathcal{C}}=\mathcal{L},
\]
one finds that
\[
\|\mathcal{S}u(x, (n-k)\tau)-u(x, (n-k+1)\tau)\|_{\infty}\le C\tau^2,
\]
with $C$ independent of $N$. 

Hence,
\[
|\mathcal{S}^n\phi(x)-u(x, n\tau)|\le \sum_{k=1}^n \|\mathcal{S}u(x, (n-k)\tau)-u(x, (n-k+1)\tau)\|_{\infty}\le CN\tau^2\le C\tau.
\] 
The claim is then proved.
\end{proof}

\begin{remark}
This result is reminiscent of the results by Golse et al. \cite{golse2019random}, where the average of the one marginal density matrices for $N$ body quantum system has been used for the Random Batch Method, and the convergence rate is also $O(\tau)$ under a certain weak norm.
\end{remark}

The backward equation for the Random Batch system is given by
\begin{gather}\label{eq:backwardrbm}
\partial_t \tilde{u}= \sum_{i=1}^N \left[b(x_i)+\frac{1}{p-1}\sum_{j: j\neq i}I_{ij}^{(k-1)} m_j K_{ij}(x_i, x_j)\right]\cdot\nabla_{x_i}
\tilde{u}+\sum_i \frac{1}{2}\sigma^2\Delta_{x_i}\tilde{u}.
\end{gather}

Hence, Theorem \ref{thm:weakconvergence} in fact says the following result when we apply random batch method to backward equations or Liouville equations ($\sigma=0$).
\begin{theorem}
Consider the backward equation \eqref{eq:backward} and its corresponding equation using Random Batch Method \eqref{eq:backwardrbm}. If the initial value is given by
\begin{gather}
u|_{t=0}=\tilde{u}|_{t=0}=\frac{1}{N}\sum_i \omega_i \varphi(x_i),
\end{gather}
then it holds that
\begin{gather}
\sup_{t\le T}\|u(t)-\tilde{u}(t)\|_{\infty}\le C(T)\tau.
\end{gather}
\end{theorem}

\begin{remark}
For general initial data, the approximation in $L^{\infty}$ given by random batch method for backward equation (Liouville equation when $\sigma=0$) can not be uniform in $N$.
\end{remark}

\begin{remark}
The situation can be different if one considers the Liouville equation of second order system for the density distribution \cite{jin2005hamiltonian}. Clearly, in this case, we cannot use the $L^{\infty}$ norm to gauge the difference. Instead, a certain weak norm for the combination of one marginals should be considered as in \cite{golse2019random}. This is left for future.
\end{remark}

\section*{Acknowledgement}
S. Jin was partially supported by the NSFC grant No. 31571071.  The work of L. Li was partially sponsored by NSFC 11901389, Shanghai Sailing Program 19YF1421300 and NSFC 11971314. The work of J.-G. Liu was partially supported by KI-Net NSF RNMS11-07444 and NSF DMS-1812573.

\appendix

\section{Proof of Lemma \ref{lmm:conditionalest}}\label{app:conditionalmoments}

\begin{proof}
The first part is similar as in \cite{jin2020random} except that for $b(\cdot)$ terms, we need to use the one-sided Lipschitz condition. Hence, we omit the proof.

We now prove \eqref{eq:conditionalincrements}.
Consider a realization so that the equation is written as
\[
d\tilde{X}^i(t)=b(\tilde{X}^i)\,dt+\frac{1}{p-1}\sum_{j\in \mathcal{C}_{\theta},j\neq i}m_j K_{ij}(\tilde{X}^i, \tilde{X}^{j})
+\sigma dB^i,
\]
where $\mathcal{C}_{\theta}$ again  is the random batch that contains $i$ from the random division at $t_{k-1}$, or $\mathcal{C}^{(k-1)}$.  It follows that
\begin{gather*}
\begin{split}
\mathbb{E}\Big(\tilde{X}^i(t)-\tilde{X}^i(t_{k-1}) \Big| \mathcal{F}_{k-1}\Big)
=& \int_{t_{k-1}}^t\mathbb{E}(b(\tilde{X}^i)| \mathcal{F}_{k-1})\,ds \\
 &+\int_{t_{k-1}}^t\mathbb{E}\left( \frac{1}{p-1}\sum_{j\in \mathcal{C}_{\theta},j\neq i}m_j K_{ij}(\tilde{X}^i,\tilde{X}^{j}) \Big| \mathcal{F}_{k-1}\right)\,ds.
\end{split}
\end{gather*}
Note that $K$ is bounded and $|b(x)|\le C(1+|x|^q)$ for some $q>0$.   
Together with \eqref{eq:conditionmoment}, this implies the first estimate in  \eqref{eq:conditionalincrements}.

For the second equation in \eqref{eq:conditionalincrements}, It\^o's formula implies that
\begin{multline*}
\frac{d}{dt}\mathbb{E}\left[|\tilde{X}^i(t)-\tilde{X}^i(t_{k-1})|^2|\mathcal{F}_{k-1} \right] \\
=2\mathbb{E}\left[ \Big(\tilde{X}^i(t)-\tilde{X}^i(t_{k-1})\Big)\cdot \left(b(\tilde{X}^i)+\frac{1}{p-1}\sum_{j\in \mathcal{C}_{\theta},j\neq i}m_jK_{ij}(\tilde{X}^i,\tilde{X}^{j}) \right)
\Big|\mathcal{F}_{k-1} \right]
+\sigma^2.
\end{multline*}

Note that
\begin{multline*}
\Big(\tilde{X}^i(t)-\tilde{X}^i(t_{k-1})\Big)\cdot b(\tilde{X}^i)
\le \beta |\tilde{X}^i(t)-\tilde{X}^i(t_{k-1})|^2
+\Big(\tilde{X}^i(t)-\tilde{X}^i(t_{k-1})\Big)\cdot b(\tilde{X}^i(t_{k-1}))\\
\le C|\tilde{X}^i(t)-\tilde{X}^i(t_{k-1})|^2+C(1+|\tilde{X}^i(t_{k-1})|^{q})
\end{multline*}

Hence,
\[
\frac{d}{dt}\mathbb{E}\left[|\tilde{X}^i(t)-\tilde{X}^i(t_{k-1})|^2|\mathcal{F}_{k-1} \right]
\le C\mathbb{E}\left[|\tilde{X}^i(t)-\tilde{X}^i(t_{k-1})|^2|\mathcal{F}_{k-1} \right]+C(1+|\tilde{X}^i(t_{k-1})|^{q}).
\]
The claim then follows.

\end{proof}

\section{Proof of Proposition \ref{pro:boundofderivatives}}\label{app:uderi}

{\bf Step 1--} Estimates of $\nabla_{x_i}u$.

Taking $\nabla_{x_{\ell}}$ in \eqref{eq:backward}, one has
\begin{gather}
\partial_t \nabla_{x_{\ell}}u=\mathcal{L}\nabla_{x_{\ell}}u+f_{\ell}(x)
\end{gather}
with
\begin{multline}\label{eq:fl}
f_{\ell}(x,t)=\left[\nabla_{x_{\ell}}b(x_{\ell})+\frac{1}{N-1}\sum_{j=1}^N
 m_j\nabla_{x_{\ell}}K_{\ell j}(x_{\ell}, x_j) \right]\cdot\nabla_{x_{\ell}}u \\
+\frac{1}{N-1}\sum_{j=1}^N m_{\ell}\nabla_{x_{\ell}}K_{j\ell}(x_j, x_{\ell})\cdot \nabla_{x_j}u.
\end{multline}
Here $\nabla_{x_{\ell}}K$ is a second order tensor and we 
use the convention that $(A\cdot v)_i:=\sum_j A_{ij}v_j$.

Hence,
\[
\nabla_{x_{\ell}}u(t)
=e^{t\mathcal{L}}\nabla_{x_{\ell}}u(0)
+\int_0^te^{(t-s)\mathcal{L}}f_{\ell}(x,s)\,ds.
\]
Since the semigroup $e^{t\mathcal{L}}$ is nonexpansive in $L^{\infty}$ as in \eqref{eq:semigroup}, we find
\[
\|\nabla_{x_{\ell}}u(t)\|_{\infty}
\le \|\nabla_{x_{\ell}}u(0)\|_{\infty}
+\int_0^t\|f_{\ell}(x,s)\|_{\infty}\,ds.
\]
However, by \eqref{eq:fl}, the linear transform from $[\nabla_{x_1}u, \cdots, \nabla_{x_N}u]$ to $[f_1, \cdots, f_N]$ is a block matrix with 
the $L^{\infty}$ norm bounded. This is because all the off-diagonal blocks are of order $O(\frac{1}{N})$. Hence,
\[
\|f_{\ell}(x,s)\|_{\infty}\le C\max_{\ell}\|\nabla_{x_{\ell}}u(t)\|_{\infty}=:a(t).
\]
Therefore,
\[
a(t)\le a(0)+C\int_0^ta(s)\,ds.
\]
This means
\[
\sup_{t\le T}a(t)\le C(T)a(0).
\]
Clearly since
\[
u(0)=\frac{1}{N}\sum_i \omega_i \varphi(x_i)
\Rightarrow \|\nabla_{x_i}u(0)\|_{\infty}\le C\frac{1}{N},
\]
the estimate for $\nabla_{x_i}u$  follows by Gr\"onwall's lemma.

{\bf Step 2--} Second order derivatives

We first compute the $\nabla_{x_i}^2$ derivatives (the $x_i$ Hessian).

\begin{gather}
\partial_t\nabla_{x_{\ell}}^2 u
=\mathcal{L}\nabla_{x_{\ell}}^2 u
+\tilde{f}_{\ell},
\end{gather}
with
\begin{gather}
\begin{split}
\tilde{f}_{\ell}= &2\nabla_{x_{\ell}}b\cdot\nabla_{x_{\ell}}^2u
+\frac{1}{N-1}\sum_j 2m_j \nabla_{x_{\ell}}K_{\ell j}(x_{\ell}, x_j)\cdot \nabla_{x_{\ell}}^2 u  \\
&+\frac{1}{N-1}\sum_{i: i\neq \ell} 2m_{\ell} \nabla_{x_{\ell}}K_{i \ell}(x_i, x_{\ell})\cdot \nabla_{x_i}\nabla_{x_{\ell}} u\\
&+
\left(2\nabla_{x_{\ell}}b+\frac{1}{N-1}\sum_j m_j\nabla_{x_{\ell}}^2K_{\ell j}(x_{\ell}, x_j)\right)\cdot\nabla_{x_{\ell}}u\\
&+\frac{1}{N-1}\sum_i m_{\ell}
\nabla_{x_{\ell}}^2K_{i\ell}(x_i, x_{\ell})\cdot\nabla_{x_i}u.
\end{split}
\end{gather}

Again, $\nabla_{x_i} b$ for a vector field $b$ is a second order tensor with $(\nabla_{x_i} b)_{rs}=\partial_{x_i^{(r)}}b^{(s)}$.  For the dot product between tensors: $A\cdot B$ means the contraction between the last index of $A$ and the first index of $B$. For example, we use the convention that
\[
(\nabla_{x_{\ell}}K_{i \ell}(x_i, x_{\ell})\cdot \nabla_{x_i}\nabla_{x_{\ell}} u)_{rs}
:=\sum_{q=1}^d \partial_{x_{\ell}^{(r)}}K_{i\ell}^{(q)}\partial_{x_i^{(q)}}
\partial_{x_{\ell}^{(s)}}u.
\]
Thus,
\begin{gather*}
\begin{split}
\tilde{f}_{\ell} =&\tilde{A}_{\ell}^1\cdot \nabla_{x_{\ell}}^2u
+\frac{1}{N-1}\sum_i \tilde{A}_{i\ell}^2\cdot\nabla_{x_i}\nabla_{x_{\ell}}u\\
&+\tilde{A}_{\ell}^3\cdot\nabla_{\ell} u
+\frac{1}{N-1}\sum_i \tilde{A}_{i\ell}^4\cdot\nabla_{x_i}u.
\end{split}
\end{gather*}
Here, all the $\tilde{A}$ tensors are bounded independent of $N$. There are $O(N)$ terms in the summation in the first line, and thus the linear transform  is again bounded in $L^{\infty}\to L^{\infty}$ independent of $N$.
The terms on the second line is clearly controlled 
as $C\frac{1}{N}$ by the estimates of $\nabla_{x_i}u$.

Similarly, one can compute for $\ell\neq q$ that
\begin{gather}\label{eq:mixedsecondderi}
\partial_t\nabla_{x_{\ell}}\nabla_{x_q}u
=\mathcal{L}\nabla_{x_{\ell}}\nabla_{x_q}u
+g_{\ell q}
\end{gather}
where the expression of $g_{\ell q}$ is complicated but it is of the following form
\begin{multline}\label{eq:glq}
g_{\ell q}=A_{\ell m}^1\cdot \nabla_{x_{\ell}}\nabla_{x_q}u
+\frac{1}{N-1}\sum_j (A_{j\ell}^2\cdot\nabla_{x_j}\nabla_{x_{\ell}}u
+A_{j q}^3\cdot\nabla_{x_j}\nabla_{x_{q}}u)\\
+\frac{1}{N-1}(m_{\ell}\nabla_{x_{\ell}}\nabla_{x_q}K_{\ell q}(x_{\ell}, x_q)\cdot\nabla_{\ell} u
+m_q\nabla_{x_q}\nabla_{x_{\ell}}K_{q\ell}(x_q, x_{\ell})\cdot\nabla_{x_q}u ).
\end{multline}
Here, $A^k$'s are second order tensors that are made up of the derivatives of $b, K_{ij}$, and thus bounded by constants independent of $N$. Note that there are $O(N)$ terms in the summation in the first line, and therefore the linear transform is again bounded in $L^{\infty}\to L^{\infty}$ independent of $N$.
The terms on the second line is clearly controlled 
as $C\frac{1}{N^2}$ by the estimates of $\nabla_{x_i}u$.

We first of all consider all the second order derivatives $\nabla_{x_{\ell}}\nabla_{x_q}$ where $\ell$ can be equal or not equal to $q$ .
By similar argument as for $\nabla_{x_i}u$, one can get the estimate
\begin{gather}
\|\nabla_{x_{\ell}}\nabla_{x_q}u(t)\|_{\infty}
\le \|\nabla_{x_{\ell}}\nabla_{x_q}u(0)\|_{\infty}
+C\int_0^t \max_{\ell, q} \|\nabla_{x_{\ell}}\nabla_{x_q}u(s)\|_{\infty}\,ds
+C(T)\frac{1}{N}.
\end{gather}
Moreover, 
\[
\|\nabla_{x_{\ell}}^2u(0)\|_{\infty}\le C\frac{1}{N}, ~~~
\|\nabla_{x_{\ell}}\nabla_{x_q}u(0)\|_{\infty}=0.
\]
Then, Gr\"onwall's inequality tells us that
\begin{gather}
\max_{\ell, q}\|\nabla_{x_{\ell}}\nabla_{x_q}u(t)\|_{\infty}\le C(T)\frac{1}{N}.
\end{gather}

Now, we focus on \eqref{eq:mixedsecondderi} for $\ell\neq q$.
To do this, we need to separate out the $j=\ell$ and $j=q$ terms
in \eqref{eq:glq}, which are of order $O(\frac{1}{N})$, and thus with the prefactor $1/(N-1)$, they contribute $O(1/N^2)$ to $g_{\ell q}$.
Hence, one has for $\ell\neq q$ that
\begin{gather}
\|\nabla_{x_{\ell}}\nabla_{x_q}u(t)\|_{\infty}
\le \|\nabla_{x_{\ell}}\nabla_{x_q}u(0)\|_{\infty}
+C\int_0^t \max_{\ell\neq q} \|\nabla_{x_{\ell}}\nabla_{x_q}u(s)\|_{\infty}\,ds
+C(T)\frac{1}{N^2}.
\end{gather}
Since $\|\nabla_{x_{\ell}}\nabla_{x_q}u(0)\|_{\infty}=0$ for $\ell\neq q$,
one then has
\begin{gather}
\max_{\ell\neq q}\|\nabla_{x_{\ell}}\nabla_{x_q}u(t)\|_{\infty}
\le C\int_0^t \max_{\ell\neq q} \|\nabla_{x_{\ell}}\nabla_{x_q}u(s)\|_{\infty}\,ds
+C(T)\frac{1}{N^2}.
\end{gather}
Gr\"onwall's lemma then tells us that
\begin{gather}
\max_{\ell\neq q}\|\nabla_{x_{\ell}}\nabla_{x_q}u(t)\|_{\infty}\le C\frac{1}{N^2}.
\end{gather}

{\bf Step  3--} Higher order derivatives.

The higher order derivatives can be similarly estimated using induction.
For these proofs, some derivatives that are not listed in Proposition \ref{pro:boundofderivatives} should be involved. For example, for third order derivatives, one should expect
\[
\|\nabla_{x_i}\nabla_{x_j}\nabla_{x_k}u\|_{\infty}=O(1/N^3)
\]
for distinct $i,j,k$. These proofs are tedious but the essential ideas are the same as we prove the claims for the Hessian. We omit the details.

\bibliographystyle{plain}
\bibliography{sdealg}

\end{document}